\documentclass[12pt]{amsart}

\usepackage{amsfonts}
\usepackage{amssymb}
\usepackage{amsmath}
\usepackage{amscd}
\usepackage{amstext}
\usepackage{ifthen}
\usepackage[all]{xy}
\usepackage{enumerate}
\usepackage[pdftex]{color}

\makeatletter
\def\@cite#1#2{{\m@th\upshape\bfseries%
[{#1\if@tempswa{\m@th\upshape\mdseries, #2}\fi}]}}
\makeatother

\newtheorem{theorem}{Theorem}[section]
\newtheorem{lemma}[theorem]{Lemma}
\newtheorem{corollary}[theorem]{Corollary}

\theoremstyle{definition}
\newtheorem{definition}[theorem]{Definition}
\newtheorem{remark}[theorem]{Remark}

\numberwithin{equation}{section}

\newcommand{\bB}{{\mathbb{B}}}
\newcommand{\bC}{{\mathbb{C}}}

\newcommand{\bZ}{{\mathbb{Z}}}

  \newcommand{\B}{{\mathcal{B}}}
  \newcommand{\C}{{\mathcal{C}}}

  \newcommand{\F}{{\mathcal{F}}}
  \newcommand{\G}{{\mathcal{G}}}
\renewcommand{\H}{{\mathcal{H}}}

\renewcommand{\L}{{\mathcal{L}}}
  \newcommand{\M}{{\mathcal{M}}}

  \newcommand{\T}{{\mathcal{T}}}
  \newcommand{\U}{{\mathcal{U}}}
  \newcommand{\V}{{\mathcal{V}}}
  
  \newcommand{\X}{{\mathcal{X}}}

\renewcommand{\phi}{\varphi}
\newcommand{\upchi}{{\raise.35ex\hbox{\ensuremath{\chi}}}}

\newcommand{\ft}{{\mathfrak{t}}}

\newcommand{\rC}{{\mathrm{C}}}

\newcommand{\qand}{\quad\text{and}\quad}

\newcommand{\qfor}{\quad\text{for}\quad}

\newcommand{\AND}{\text{ and }}
\newcommand{\FOR}{\text{ for }}

\newcommand{\OR}{\text{ or }}

\newcommand{\id}{{\operatorname{id}}}

\newcommand{\spn}{\operatorname{span}}

\newcommand{\re}{\operatorname{Re}}

\newcommand{\Repn}{\operatorname{Rep}_{\T_2}}
\newcommand{\Rep}{\operatorname{Rep}}

\newcommand{\supp}{\operatorname{supp}}

\newcommand{\bsl}{\setminus}

\newcommand{\ip}[1]{\langle #1 \rangle}
\newcommand{\mt}{\varnothing}
\newcommand{\ol}{\overline}

\begin{document}

\title[Topological graphs]{Isomorphisms of tensor algebras of topological graphs}

\author[K. R. Davidson]{Kenneth R. Davidson}
\address{Pure Mathematics Department\\
University of Waterloo\\
Waterloo, ON\; N2L--3G1\\
CANADA}
\email{krdavids@uwaterloo.ca}

\author[J. Roydor]{Jean Roydor}
\address{D\'epartement de Math\'ematiques. \\
Universit\'{e} d'Orleans \\
45067 Orleans Cedex 2 \\
FRANCE}
\email{jean.roydor@univ-orleans.fr}

\thanks{First author partially supported by an NSERC grant.}

\begin{abstract}
We show that if two tensor algebras of topological graphs are algebraically
isomorphic, then the graphs are locally conjugate.  Conversely, if the base
space is at most one dimensional and the edge space is compact,
then locally conjugate topological graphs yield completely
isometrically isomorphic tensor algebras.
\end{abstract}

\date{}
\maketitle

\section{Introduction}

In \cite{K}, Katsura defined the notion of a topological graph
which unifies countable directed graphs and dynamical systems.
This enabled him to introduce a new class of C*-algebras
generalizing the classes of graph C*-algebras and
homeomorphism C*-algebras.
These C*-algebras arise as Cuntz-Pimsner algebras over
C*-correspondences, a construction due to Pimsner \cite{Pim}.
Muhly and Solel have a long series of papers which consider the
non-selfadjoint tensor algebra associated to a C*-corrsepondence
(see \cite{MS1,MS2} for a start) which develop a detailed theory,
making this a very tractable, yet rich, family of operator algebras.

This class of tensor algebras contains two well-known classes of operator
algebras as special cases: tensor algebras of countable directed graphs
and operator algebras associated to dynamical systems in one or several
variables.  One important direction is the attempt to classify these algebras,
and in particular to decide to what extent one can recover the graph or
dynamical system from the operator algebra.  We continue this program in
this paper in the context of topological graphs.

This was successfully done for countable directed graphs, showing that
the graph can be completely recovered from the operator algebra.  This
was accomplished in increasing generality by Kribs and Power \cite{KP},
Solel \cite{Sol} and Katsoulis and Kribs \cite{KK}.

Operator algebras of dynamical systems has a very long history.
The story for nonself-adjoint algebras begins with the work of Arveson
\cite{Ar,AJ} where an operator algebra is associated to a dynamical system
determined by a single map from a compact Hausdorff space to itself
subject to certain additional hypotheses.
Peters \cite{Pet} showed how to define an abstract operator algebra,
known as a semicrossed product, to any such discrete dynamical system,
and classified them with weaker constraints.  Further results by
Hadwin and Hoover \cite{HH}, Power \cite{Pow} and Davidson and
Katsoulis \cite{DK2} eventually showed that the dynamical system can be
completely recovered from the operator algebra up to conjugacy for any
proper continuous map of a locally compact Hausdorff space to itself.

In the case of multivariable dynamics, Davidson and Katsoulis \cite{DK}
define a universal operator algebra that encodes $n$ proper continuous
maps of a locally compact Hausdorff space $X$ to itself.  This algebra is
the tensor algebra of a C*-correspondence.  It is shown that isomorphic
algebras have dynamical systems which are piecewise conjugate, which
is a local version of conjugacy that allows different permutations of the
maps in different neighbourhoods.
A strong converse, showing that piecewise conjugate systems determine
completely isometrically isomorphic algebras, was established if
$\dim X \le 1$ or if the number of maps was at most three.

In the case of topological graphs, it is necessary to formulate an appropriate
notion of piecewise conjugacy, that we call local conjugacy.
In this paper, we show that if two tensor algebras of topological graphs
are algebraically isomorphic, then the topological graphs are locally conjugate.
The proof of this result uses the same ingredients as in the multivariable
dynamics case: characters and nest representations of the tensor algebras.
However the converse proves to be considerably harder. This is because
the topological graphs that arise in \cite{DK} are topologically much simpler
than the general case. We prove that the strong converse holds when the
edge space is compact and the base space has dimension at most one.

\section{Preliminaries}

A topological graph $E=(E^0,E^1,r,s)$ consists of two locally
compact spaces $E^0$, $E^1$, a continuous proper mapping
$r:E^1 \to E^0$ and a local homeomorphism $s:E^1 \to E^0$.
The set $E^0$ is called the base space and $E^1$ the edge space.
A topological graph is called compact if both of these spaces are
compact Hausdorff spaces.
An important class of examples of topological graphs are
multivariable dynamical systems.
More precisely, if $X$ is a locally compact Hausdorff space
endowed with $n$ proper continuous selfmaps
$(\sigma_1,\dots,\sigma_n)$, then let
$E^0=X$, $E^1=\{1,\dots,n\} \times X$,
let $s$ be the natural projection onto $E^1$ onto $E^0$,
and $r=\cup_i \sigma_i$ the union the maps $\sigma_i$.
This defines a topological graph.
More examples of topological graphs can be found in \cite{K2}.

Following \cite{K}, we consider the associated C*-correspondence
$\X(E)$ over $\rC_0(E^0)$.
We recall that the right and left actions of $\rC_0(E^0)$ on
$\rC_c(E^1)$ are given by
\[
 (f \cdot x \cdot g)(e)=f(r(e))x(e)g(s(e))
\]
for $x \in \rC_c(E^1)$, $f,g \in \rC_0(E^0)$ and $e \in E^1$.
The inner product is defined for $x,y \in \rC_c(E^1)$ by
\[
 \ip{ x, y } (v) = \sum_{ e \in s^{-1}(v)} \ol{x(e)}y(e)  \qfor v \in E^0.
\]
Finally, $\X(E)$ denotes the completion of $\rC_c(E^1)$ for the norm
\[
 \Vert x \Vert=\sup_{v \in E^0} \langle x, x \rangle (v)^{1/2} .
\]
See \cite{R} for more details.

Looking at the C*-correspondence in \cite{DK}, it is easy to see
that the tensor algebra of a multivariable dynamical system
$(X;\sigma_1,\dots,\sigma_n)$ coincides with the tensor algebra
of the topological graph described above.

In general, the tensor algebra of a C*-correspondence can be defined
by a certain universal property \cite{MS1}.
In our case, this tensor algebra is the universal algebra for pairs
of representations $(\pi,t,\H)$ where $\pi:\rC_0(E^0) \to \B(\H)$ is a
$*$-homomorphism and $t:\X(E) \to \B(\H)$ is a completely
contractive map satisfying \textit{covariance relations}
\[
 t(f \cdot x)=\pi(f)t(x) \qand t(x \cdot f)=t(x)\pi (f).
\]
More precisely, $\T (E)_+$ is the unique operator algebra containing
$\rC_0(E^0)$ and $\X(E)$ such that every such pair of representations
$(\pi,t,\H)$ extends uniquely to a completely contractive representation
of $\T (E)_+$.
We will use repeatedly this universal property to construct characters and nest representations of $\T (E)_+$.

The tensor algebra associated to a C*-correspondence can also be
realized concretely as a subalgebra of the adjointable operators on
the related Fock space (see \cite{K}).
We recall briefly this construction in our case.
Let $E=(E^0,E^1,r,s)$ be a topological graph and let $\X(E)$ be
the associated C*-correspondence over $\rC_0(E^0)$.
We define
\[
 \F(E) = \rC_0(E^0) \oplus \bigoplus_{n \geq 1} \X(E)^n
\]
to be the Fock space of $E$,
where $\X(E)^n$ denotes the $n$-fold tensor product of $\X(E)$.
Let $\L(\F(E))$ be the C*-algebra of adjointable operators on $\F(E)$.
The Fock representation $(\pi_\infty,t_\infty)$ is defined in the following way:
for $f \in \rC_0(E^0)$ and $x_1 \otimes \dots \otimes x_n \in \X(E)^n$, set
\[
 \pi_\infty(f)(x_1 \otimes \dots \otimes x_n)=(f \cdot x_1) \otimes \dots \otimes x_n .
\]
This is a $*$-isomorphism from $\rC_0(E^0)$ into $\L(\F(E))$.
Also for $x \in \X(E)$ and $x_1 \otimes \dots \otimes x_n \in \X(E)^n$, let
\[
 t_\infty(x)(x_1 \otimes \dots \otimes x_n) =
 x \otimes x_1 \otimes \dots \otimes x_n .
\]
This is a completely contractive mapping from $\X(E)$ into $\L(\F(E))$.
It is not difficult to check that the pair $(\pi_\infty,t_\infty)$ satisfies the
covariance relations.
The tensor algebra associated to the topological graph $E$ is
\[
 \T (E)_+ = \ol{\spn} \{ t_\infty ^n(x) : n \geq 0,~x \in \X(E)^n \} \subset \L(\F(E))
\]
with the conventions $\X(E)^0=\rC_0(E^0)$, $t_\infty ^0=\pi_\infty$
and $t_\infty^n$ denotes the $n$-fold tensor product of $t_\infty$.

\section{Identifying the base space}

Following \cite{DK}, we wish to identify the character space $\M(E)$ of $\T (E)_+$
and the maximal analytic subsets.
Recall that an analytic subset of $\M(E)$ is the image of an injective function $\Phi$
from a domain $\Omega$ in $\bC^k$ into $\M(E)$ such that for each $A \in \T(E)_+$,
the function $\Phi(z)(A)$ is analytic on $\Omega$.  It is a maximal analytic set
if it is not contained in any larger analytic set.

Since $\T (E)_+$ contains $\rC_0(E^0)$, the restriction of any character to this
subalgebra yields a point evaluation $\delta_v$ for some $v\in E^0$.
Let $\M(E)_v$ denote those characters which extend $\delta_v$.
Recall that there is a canonical conditional expectation $\Xi$ of $\T (E)_+$ onto
$\rC_0(E^0)$ obtained by integrating the gauge automorphisms.
This map is a homomorphism, and thus $\theta_{v,0} := \delta_v \Xi$ is always
a character.  So $\M(E)_v$ is never empty.

For each $v \in E^0$, we denote
\[
E^1_{v} = \{e \in E^1 : r(e)=v = s(e) \}.
\]
This is a compact set because $r$ is proper, and is
discrete because $s$ is a local homeomorphism.
Hence it is finite.

\begin{lemma} \label{L:1}
Let $x \in \X(E)$ such that $\supp(x) \cap E^1_{v}=\mt$.
Then $\theta(x)=0$ for every $\theta \in \M(E)_v$.
\end{lemma}

\begin{proof}
Note that $\supp(x) \cap s^{-1}(v)$ is finite; and so its image under
$r$ is compact and disjoint from $v$.
Thus there is a neighbourhood $U_1$ of $\supp(x) \cap s^{-1}(v)$
with compact closure so that $r(\ol{U_1})$ is disjoint from $v$.
Therefore there is a positive function $g_1$ in $\rC_0(E^0)$
so that $g_1r|_{\ol{U_1}} = 1$ and $g_1(v)=0$.
Let $x_1 = g_1\cdot x$ and $x_2 = x-x_1$.

Observe that $\supp(x_2) \cap U_1 = \mt$.
So $s(\supp(x_2))$ is compact and disjoint from $v$.
Therefore there is a positive function $g_2$ in $\rC_0(E^0)$
so that $g_2s|_{\supp(x_2)} = 1$ and $g_2(v)=0$.
Note that $x_2 = x_2 \cdot g_2$.

For any $\theta \in \M(E)_v$,
\[
 \theta(x)=\theta(g_1 \cdot x_1) + \theta(x_2\cdot g_2) =
 \delta_v(g_1)\theta(x_1) + \theta(x_2) \delta_v(g_2) = 0 . \qedhere
\]
\end{proof}

\begin{theorem}\label{T:characters}
For $v\in E^0$, let $n$ denote the (finite) cardinality of $E^1_{v}$.
If $n=0$, then $\M(E)_v=\{ \theta_{v,0} \}$.
If $n \geq 1$, then $\M(E)_v$ is homeomorphic to $\ol{\bB}_n$,
the closed unit ball of $\bC^n$, and the image of the open ball
$\bB_n$ is a maximal analytic set.
Moreover, these are the only maximal analytic sets.
\end{theorem}

\begin{proof}
Let $E^1_{v}=\{e_i : 1 \le i \le n \}$.
Since $E^1_v$ is a finite set, we can select disjoint neighbourhoods
$e_i \in V_i \subset \ol{V_i} \subset U_i$ of $e_i$ such that $\ol{U_i}$ is compact and
such that $s$ maps each $U_i$ homeomorphically onto a neighbourhood
$U$ of $v$.
Pick contractions $\ft_i \in \rC_c(E^1) \subset \X(E)$ with compact support
in $U_i$ such that $\ft_i|_{V_i} =1$.
Observe that the row operator $\ft = \big[ \ft_1\ \dots\  \ft_n \big]$ is a contraction.
Consider the map $\kappa$ defined on $\M(E)_v$ by
\[ \kappa(\theta)=(\theta(\ft_1),\dots,\theta(\ft_n)).\]
As a character $\theta$ is completely contractive, the range of $\kappa$
is contained in $\ol{\bB}_n$.

To prove that $\kappa$ is surjective, let $z=(z_1,\dots, z_n) \in \ol{\bB}_n$,
and consider the complete contraction $t_z$ defined on $\X(E)$ by
\[ t_z(x)=\sum_{i=1}^n z_i x(e_i). \]
Then $(\delta_v,t_z,\bC)$ is covariant representation.
By the universal property of $\T (E)_+$, it extends uniquely to a
character $\theta_{v,z} \in \M(E)_v$.
Clearly $\kappa(\theta_{v,z})=z$.

For the injectivity, let $\theta_1, \theta_2 \in \M(E)_v$ such that
$\kappa(\theta_1)=\kappa(\theta_2)$.
As $\T (E)_+$ is generated by $\rC_0(E^0)$ and $\X(E)$,
it suffices to prove that $\theta_1$ and $\theta_2$ agree on $\X(E)$.
Let $x \in \X(E)$ and pick an open set $U$ of $E^1$ with compact closure such that
$E^1_v \cap \ol{U}=\mt$ and $\{U, V_0,V_1,\dots, V_n\}$
forms an open cover of $\supp(x)$.
Let $\{ p, p_i : 1 \le i \le n\}$ be a partition of the unity in
$\rC_c(E^1) \subset \X(E)$ associated to this open cover.
Then $x = xp + xp_1  + \cdots + xp_n$.
By Lemma \ref{L:1},
\[ \theta_1(xp)=\theta_2(xp)=0. \]
Moreover for each $i \ge 1$, $\supp(xp_i) \subset V_i$ so
$xp_i=\ft_i \cdot b_i$, where $b_i$ is a function in $\rC_0(E^0)$
extending $xp_i \circ s|_{V_i}^{-1}$.
Hence
\[
 \theta_1(xp_i)=\theta_1(\ft_i)b_i(v)=\theta_2(\ft_i)\theta_2(b_i)=\theta_2(xp_i) .
\]
This is true for each $i$, so $\theta_1=\theta_2$.

Thus $\kappa$ is a bijection.
We denote the character $\kappa^{-1}(z)$ as $\theta_{v,z}$.
Now for $x \in \X(E)$, $\theta_{v,z}(x) = t_z(x)$ is a linear function.
Thus the extension to $\sum_{k\ge0} t^n_\infty(\X(E)^n)$ takes polynomial values.
A sequence of such elements converging to $T\in \T(E)_+$ converges uniformly
on $\M(E)$ because characters are (completely) contractive.
It follows that the function $\hat T(z) := \theta_{v,z}(T)$ is analytic on the
relative interior $\M^0_v$ of $\M_v$, which is identified with the open ball $\bB_n$.
Thus this is an analytic set.

Finally if $U$ is an analytic subset of $\M(E)$, then the restriction to $\rC_0(E^0)$
is analytic, which forces it to lie in some $\M_v$.  Thus it is clear that it is a subset of
$\M^0_v$.  So we have identified the maximal analytic sets.
\end{proof}

As in \cite{DK}, we observe that the quotient space of $\M(E)$ obtained by
mapping the closure of each maximal analytic set to a point is homeomorphic to $E^0$.
Hence we deduce the following important corollary.

\begin{corollary}\label{C:base space}
The character space $\M (E)$ determines $E^0$ up to homeomorphism.
\end{corollary}

\section{Local conjugacy}

For $v,w$ in $E^0$, we denote
\[
 E^1_{v,w}=\{e \in E^1 : r(e)=v \AND s(e)=w \} = r^{-1}(v) \cap s^{-1}(w) .
\]
This is a finite set because $r^{-1}(v)$ is compact, and $s^{-1}(w)$ is a
discrete set.
Denote its cardinality by  $n_{v,w}$.

Recall that a nest representation is a homomorphism of the algebra into
the $2\times 2$ upper triangular matrices with respect to a basis $e_1,e_2$
such that the only proper invariant subspace is $\bC e_1$.
The two diagonal entries are characters, and the $1,2$ entry is a derivation.
Let $\Repn(E)$ denote the set of all $2$-dimensional nest representations of
$\T (E)_+$.
Let $\Rep_{v,w}(E)$ be the set of $2$-dimensional nest representations of
$\T (E)_+$ such that the $1,1$ entry belongs to $\M(E)_v$
and the $2,2$ entry belongs to $\M(E)_w$.
We also denote $\Rep^d_{v,w}(E)$ the subset of $\Rep_{v,w}(E)$ of
representations which are diagonal when restricted to $\rC_0(E^0)$.

\begin{lemma}\label{L:key}
$\Rep_{v,w}(E)$ is non-empty if and only if $E^1_{v,w} \neq \mt$.
Moreover, if $x \in \X(E)$ is such that $\supp(x) \cap E^1_{v,w}= \mt$,
then $\rho(x)$ is diagonal for every $\rho \in \Rep^d_{v,w}(E)$.
\end{lemma}

\begin{proof}
First assume that $E^1_{v,w}$ is non-empty.
Let $\lambda_e$ for $e \in E^1_{v,w}$ be complex scalars
such that $0 < \sum |\lambda_e| \le 1$.
Define a $*$-homomorphism $\pi:\rC_0(E^0) \to \mathbb{M}_2$
by
\[
 \pi(f)= \begin{bmatrix} f(v)&0 \\ 0&f(w)\end{bmatrix}
\]
and a map $t:\X(E) \to \mathbb{M}_2$ by
\[
 t(x)= \begin{bmatrix}
 0& \sum_{e \in E^1_{v,w}} \lambda_e x(e) \\
 0&0  \end{bmatrix}
\]
This is contractive (as it is a convex combination of contractive maps).
Thus $t$ is completely contractive because it's $1$-dimensional.

Clearly
\[
 t(f \cdot x) = \sum \lambda_e f(r(e))x(e) = \pi(f)t(x)
\]
and
\[
 t(x \cdot f) = \sum \lambda_e x(e) f(s(e)) = t(x)\pi (f) .
\]
Hence by the universal property of the tensor algebra,
this pair of representations extends to a completely contractive representation
$\rho$ of $\T (E)_+$.
The range of $\rho$ is the entire upper triangular $2 \times 2$ matrices
because some $\lambda_e$ is non-zero.
If $x \in \X(E)$ satisfies $\supp(x) \cap E^1_{v,w} = \mt$,
then $\rho$ is diagonal.  So $\rho$ belongs to $\Rep^d_{v,w}(E)$.

Conversely suppose that $\Rep_{v,w}(E)$ is non-empty.
Then applying a similarity (as in Lemma~\ref{L:diag} below) yields an element of  $\Rep^d_{v,w}(E)$.
Let $x \in \X(E)$ such that $\supp(x) \cap E^1_{v,w}= \mt$.
We will prove that $\rho(x)$ is diagonal for every $\rho$ in $\Rep^d_{v,w}(E)$.
Arguing as in Lemma~\ref{L:1}, we can split $x = g_1\cdot x_1 + x_2 \cdot g_2$
so that $g_1(v)=0$ and $g_2(w) = 0$.
Thus $\rho(x)$ has the form
\begin{align*}
 \rho(x) &= \rho(g_1) \rho(x_1) + \rho(x_2) \rho(g_2) \\ &=
 \begin{bmatrix} 0&0 \\ 0& * \end{bmatrix} \begin{bmatrix} *&* \\ 0& * \end{bmatrix} +
 \begin{bmatrix} *&* \\ 0& * \end{bmatrix} \begin{bmatrix} *&0 \\ 0& 0 \end{bmatrix}
 = \begin{bmatrix} *&0 \\ 0& * \end{bmatrix}
\end{align*}

If $E^1_{v,w} $ were empty and $\rho \in \Rep^d_{v,w}(E)$,
then $\rho(x)$ is diagonal for every  $x \in \X(E)$.
Thus $\rho$ is not a nest representation, contrary to assumption.
Hence $E^1_{v,w}$ is not empty.
\end{proof}

We also need a technical result which is a minor modification of \cite[Lemma~3.13]{DK}
using $\sigma=r_Es_{E \vert L}^{-1}$.

\begin{lemma}\label{L:diag}
Let $E=(E^0,E^1,s,r)$ be a topological graph.
Let $L \subset E^1$, $K \subset E^0$ two compact subsets such that
$s_E|_L$ maps $L$ onto $K$ homeomorphically.
Set $\sigma = r_E s_E|_L^{-1}$.
Let $\Omega$ be a domain in $\bC^d$.
Suppose that there exists a map $\rho$ from $K \times \Omega$ to
$\Rep \T (E)_+$ such that
\begin{enumerate}[\em(a)]%
\item $\rho(x,z) \in \Rep_{\sigma(x),x}$,
\item $\rho$ is continuous in the point-norm topology,
\item for each $x \in K$, $\rho(x,z)$ is analytic in $z \in \Omega$.
\end{enumerate}
Then there exists a map $A$ from $K \times \Omega$ onto the group of
invertible upper triangular matrices, so that
\begin{enumerate}[\em(1)]
\item $A(x,z)\rho(x,z)(\cdot)A(x,z)^{-1} \in \Rep^d_{\sigma(x),x}$,
\item $A$ is continuous on $K\bsl \{ x :  \sigma(x)=x \} \times \Omega$,
\item $A(x,z) = I_2$ when $\sigma(x)=x$,
\item for each $x \in K$, $A(x,z)$ is analytic in $z \in \Omega$,
\item $\max \big\{ \| A(x,z) \|,~\| A(x,z)^{-1} \| \big\} \leq 1 + \| \rho(x,z) \|$.
\end{enumerate}
\end{lemma}

Now we define a notion which extends the notion of piecewise conjugacy
introduced in \cite{DK}.  To paraphrase, local conjugacy means that the
base spaces are homeomorphic and there is an open cover of the base spaces
which lift to local intertwining of the two graphs.
{}From now on, we will need to assume that the edge space is compact.

\begin{definition}
Let $E=(E^0,E^1,s_E,r_E)$ and $F=(F^0,F^1,s_F,r_F)$ be two compact topological graphs.
They are said to be \textit{locally conjugate} if there exists a homeomorphism
$\tau:E^0 \to F^0$ such that for every $u \in E^0$, there is a neighbourhood $U$ of $u$,
a homeomorphism $\gamma$ of $s_E^{-1}(U)$ onto $s_F^{-1}(\tau U)$ such that
$s_F \gamma = \tau s_E|_{s_E^{-1}(U)}$ and $r_F \gamma = \tau r_E|_{s_E^{-1}(U)}$.
\end{definition}

\begin{remark}
While we will only deal with compact topological graphs in this paper,
if $E^1$ and $F^1$ are not compact, we would modify the definition of
local conjugacy to say:
there exists a  homeomorphism
$\tau:E^0 \to F^0$ such that for every $v \in E^0$, and for every open set
$O$ in $E^0$ such that $\ol{O}$ is compact,
there is a neighbourhood $V$ of  $v$,
a homeomorphism $h_{V}$ from $r_E^{-1}(O) \cap s_E^{-1}(V)$ onto
$r_F^{-1}(\tau(O)) \cap s_F^{-1}(\tau(V))$ such that
$s_Fh_{V}=\tau s_E$ and $r_Fh_{V}=\tau r_E$.
\end{remark}

To prove the classification result, we need an equivalence relation between edges.
Two edges $e,e' \in E^1$ are said to be equivalent if there exist $U$
neighbourhood of $e$, $U'$ neighbourhood of $e'$ such that $s_E$
restricted to both $U$ and $U'$ is a homeomorphism and
$s_E(U)=s_E(U')$ and for every $v \in s(U)$,
$r_Es_E|_U^{-1}(v))=r_Es_E|_{U'}^{-1}(v))$.
In this case, we denote $e \sim e'$.
When there is a homeomorphism $\tau:E^0 \to F^0$, this equivalence relation
naturally extends to an equivalence relation between edges of $E^1$ and $F^1$ (relative to $\tau$).

The proof follows the lines of \cite[Theorem~3.22]{DK}
with some additional technicalities.

\begin{theorem}\label{T:main}
Let $E=(E^0,E^1,s_E,r_E)$ and $F=(F^0,F^1,s_F,r_F)$ be two
compact topological graphs.
If $\T (E)_+$ and $\T (F)_+$ are algebraically isomorphic, then
$E$ and $F$ are locally conjugate.
\end{theorem}

\begin{proof}
Let $\gamma$ be an isomorphism of $\T (E)_+$ onto $\T (F)_+$.
This induces a canonical map $\gamma_c$ from $\M(E)$ onto $\M(F)$
between the character spaces.
It also induces a map $\gamma_r$ from $\Repn(E)$ onto $\Repn(F)$.

Since $\M(E)$ is endowed with the weak-$*$ topology, it is
easy to see that $\gamma_c$ is continuous.
Indeed, if $\theta_\alpha$ is a net in $\M(E)$ converging to $\theta$
and $B \in \T(F)_+$, then
\[
 \lim_\alpha \gamma_c \theta_\alpha (B) = \lim_\alpha \theta_\alpha(\gamma^{-1}(B))
 = \theta(\gamma^{-1}(B)) = \gamma_c \theta(B) .
\]
The same holds for $\gamma_c^{-1}$.
So $\gamma_c$ is a homeomorphism.

Observe that $\gamma_c$ carries analytic sets to analytic sets.
Indeed, if $\Theta$ is an analytic function of a domain $\Omega$
into $\M(E)$, then
\[ \gamma_c \Theta(z)(B) = \Theta(z)(\gamma^{-1}(B)) \]
is analytic for every $B \in \T(F)_+$; and thus $\gamma_c \Theta$ is analytic.
Since the same holds for $\gamma^{-1}$, it follows that $\gamma_c$
takes maximal analytic sets to maximal analytic sets.
Thus it carries their closures, $\M(E)_u$, onto sets the corresponding sets $\M(F)_v$.

By Corollary~\ref{C:base space}, $\gamma_c$ induces a homeomorphism
$\tau$ of $E^0$ onto $F^0$.
For simplicity of notation, we identify $E^0$ and $F^0$ via $\tau$,
so that $F^0=E^0$ and $\tau = \id$.

Fix $v_0 \in E^0$. Since $s_E$ and $s_F$ are local homeomorphisms
and $E^1$ and $F^1$ are compact,  both $s_E^{-1}(v_0)=\{e_1,\dots,e_p\}$
and $s_F^{-1}(v_0)=\{f_1,\dots,f_q\}$ are finite sets.
Let $\G$ denote the equivalence class of $e_1$ among $\{e_1,\dots,e_p,f_1,\dots,f_q\}$.
After relabeling
\[
 \G=\{e_1,\dots,e_k,f_1,\dots,f_l :  e_1 \sim \dots \sim e_k \sim f_1 \sim \dots \sim f_l \}.
\]
By exchanging $E$ and $F$ if necessary, we can suppose that $k \ge l$.
Our purpose is to prove that $k=l$.
So, suppose that $k > l$.

As in \cite[Lemma 1.4]{K}, there exists a neighbourhood $W$ of $v_0$,
$U_i$ disjoint neighbourhoods of $e_i$ ($1 \le i \le p$), $V_j$ disjoint
neighbourhoods of $f_j$ ($1 \le j \le q$) such that $s_E(U_i)= W = s_F(V_j)$
homeomorphically, and
\[
 s_E^{-1}(W) = \bigcup_{i=1}^p U_i
 \qand
 s_F^{-1}(W) = \bigcup_{j=1}^q V_j
\]
and 
\[
 \sigma := r_E s_E|_{U_1}^{-1}  =
 r_E s_E|_{U_i}^{-1}  =
 r_F s_F|_{V_j}^{-1} \qfor 1 \le i \le k,\ 1 \le j \le l.
\]
If $\sigma(v_0) \ne v_0$, we also choose $W$ so that
$\sigma (\ol{W}) \cap \ol{W}=\mt$.

For each $v \in W$ and $z=(z_1,\dots,z_k)$ in the $\bC^k$,
define a representation in $\Rep^d_{w,v}(E)$ by
\[
 \rho_{v,z}=
 \begin{bmatrix}
  \delta_{\sigma (v)} & \sum_{i=1}^k z_i\delta_{e_i} \\
  0 & \delta_v
 \end{bmatrix}
\]
This is a continuous family of nest representations which is analytic
in the second variable.  Here $\rho_{v,z}$ is completely contractive
if $\|z\|_2 \le 1$ and bounded by $1+\|z\|_2$ in general.

Let $\eta_{v,z} := \gamma_r \rho_{v,z} = \rho_{v,z} \circ \gamma^{-1}$.
This is a nest representation in $\Rep_{\sigma (v),v}(F)$ with
$\| \eta(x,z) \| \le (1 + \|z\|_2)\|\gamma^{-1}\|$.
Apply Lemma \ref{L:diag} to obtain a matrix valued function $A(v,z)$ such that
\[
 \Phi(v,z)=A(v,z) \eta_{v,z}A(v,z)^{-1}
 \in \Rep^d_{w,v}(F)
\]
which is continuous on $W\bsl \{ x :  \sigma(x)=x \} \times \bC^k$,
analytic in the second variable, and such that
\[ \max \big\{ \| A(x,z) \|,~\| A(x,z)^{-1} \| \big\} \leq 1 + \| \eta(x,z) \| .\]

Now for $1 \le j \le q$, pick a contraction $\ft_j \in \C_c(F^1)$
with $\supp(\ft_j) \subset V_j$ and a compact neighbourhood
$\tilde V_j \subset V_j$ of $f_j$ such that $\ft_j|_{\tilde V_j} = 1$.
Define $\psi_j(z)$ to be the $(1,2)$ entry of $\Phi(v_0,z)(\ft_j)$, and let
\[
 \Psi(z)=(\psi_1(z),\dots,\psi_q(z)) .
\]
This is an analytic function from unit ball of $\bC^k$ into $\bC^q$.

We claim that $\psi_j(z)=0$ for $j>l$.
If $j>l$, then $f_j$ is not in $\G$.
Hence there is a net $(v_\lambda) \in W$ converging to $v_0$ such that
\[
 r_F s_F|_{ V_j}^{-1}(v_\lambda) \ne \sigma(v_\lambda) .
\]
Note that
\[
 \supp(\ft_j) \cap s_F^{-1}(v_\lambda) \subset V_j \cap s_F^{-1}(v_\lambda)
 = \{ s_F|_{V_j}^{-1}(v_\lambda) \}.
\]
Therefore,
$\supp(\ft_j) \cap F^1_{\sigma(v_\lambda),v_\lambda} = \mt$.
Hence by Lemma~\ref{L:key}, $\Phi(v_\lambda,z)(\ft_j)$ is diagonal
for all $\lambda$ and all $z\in \bC^k$.

Now there are two cases.
If $\sigma(v_0) \ne v_0$, then $A$ is continuous
(since $\sigma (\ol{W}) \cap \ol{W}=\mt$).
Hence $\Phi(v,z)$ is point-norm continuous.
Taking a limit in $\lambda$, we obtain that $\Phi(v_0,z)(\ft_j)$ is diagonal,
whence $\psi_j(z)=0$.

In the second case, $\sigma(v_0) = v_0$.
Thus the diagonal elements of the range of $\Phi(v_0,z)$ are scalars.
For $z$ fixed, the function $A(v,z)$ and its inverse are bounded.
Therefore we can pass to subnet so that $A(z) := \lim_\lambda A(v_\lambda,z)$
exists in the the group of invertible upper triangular matrices.
As $\eta_{v,z}=\rho_{v,z} \circ \gamma^{-1}$ is point-norm continuous and
$A(v_0,z)=I_2$, passing to the limit we deduce that
$A(z)\eta_{v,z}(\ft_j)A(z)^{-1}$ is diagonal, hence scalar.
So $\Phi(v_0,z)(\ft_j)$ is scalar, which means that $\psi_j(z)=0$.

We may now consider $\Psi$ as a function from $\bC^k$ into $\bC^l$.
We have $\Psi(0) = 0$.  As in \cite[Proposition~3.21]{DK}, it follows
that this cannot be an isolated zero when $l<k$.
Thus there is some $z_0\ne 0$ such that $\Psi(z_0) = 0$.
As in the proof of Theorem~\ref{T:characters}, every $x \in \X(F)$
can be written as a sum $x= x_0 + \sum_{j=1}^q \ft_j \cdot f_j$
where $x_0$ vanishes on a neighbourhood of $s_F^{-1}(v_0)$
and $f_j \in \rC_0(F^0)$.
It follows from Lemma~\ref{L:key} that $\Phi(v_0,z_0)(x_0)$ is diagonal.
As each $\Phi(v_0,z_0)(\ft_j) \Phi(v_0,z_0)(f_j)$ is also diagonal,
we deduce that $\Phi(v_0,z_0)$ is diagonal---and hence is not a nest representation.
This contradiction shows that $l=k$.

Therefore we can partition $\{e_1,\dots,e_p\}$ so they agree on a common
neighbourhood $W$ of $v_0$, and these equivalence classes will paired with a
corresponding partition of $\{f_1,\dots,f_q\}$.
In particular, $q=p$.
After reordering, we may suppose that $e_i \sim f_i$ for $1 \le i \le p$.
Thus there are disjoint neighbourhoods $U_i$ of
of $e_i$ and $V_i$ of $f_i$ for $1 \le i \le p$
such that $s_E(U_i)=W=s_F(V_j)$ homeomorphically,
\[
 s_E^{-1}(W) \subset \bigcup_{i=1}^p U_i
 \qand
 s_F^{-1}(W) \subset \bigcup_{j=1}^q V_j ,
\]
and $r_Es_E|_{U_i}^{-1} = r_Fs_F|_{V_i}^{-1}$.
The map $\gamma_i = s_F|_{V_i}^{-1} s_E|_{U_i}$ is a homeomorphism
of $U_i$ onto $V_i$ so that
\[
 s_F \gamma_i = s_E|_{U_i} \qand r_F \gamma_i = r_E|_{U_i} .
\]
Hence the map $\gamma = \bigcup_{i=1}^p \gamma_i$
is the desired homeomorphism of $s_E^{-1}(W)$ onto $s_F^{-1}(W)$
establishing local equivalence of $E$ and $F$.
\end{proof}

\section{The converse}

In this section, we discuss the converse of Theorem \ref{T:main},
namely, if two topological graphs $E$ and  $F$ are locally conjugate then
are their associated tensor algebras isomorphic (as Banach algebras)?
We are able to prove this converse when $E^0$ or $F^0$ has topological
dimension less or equal to 1 and $E^1$ and $F^1$ are compact.
In this case, the tensor algebras are completely isometrically isomorphic.

We begin by constructing a particularly good open cover of $E$.

\begin{definition}
Let $E$ and $F$ be two locally conjugate compact topological graphs
with $\dim(E^0) \leq 1$ via a homeomorphism $\tau$.
An open cover $\U=(U_i)_{i \leq n}$ of $E^0$ is called \textit{admissible} if:
\begin{enumerate}[(C1)]
\item For each $i$, there is an integer $m(i)$ and disjoint open sets $U_{ik}$
 of $E^1$ for $1 \le k \le m(i)$ such that $s_E^{-1}(U_i) = \bigcup_{k=1}^{m(i)} U_{ik}$ and
 $s_E$ maps each $U_{ik}$ homeomorphically onto $U_i$.
\item For each $1\le i \le n$, there is a homeomorphism $\gamma_i$ from
 $s_E^{-1}(U_i)$ onto $s_F^{-1}(\tau U_i)$ such that
 $s_F \gamma_i = \tau s_E|_{s_E^{-1}(U_i)}$ and
 $r_F \gamma_i = \tau r_E|_{s_E^{-1}(U_i)}$.
\item $U_i \cap U_j \cap U_k=\emptyset$, whenever $i \neq j \neq k \neq i$.
\item $\ol{U}_i \bsl U_j$ and $\ol{U}_j \bsl U_i$ are disjoint for $i \ne j$.
\item If $U_i \cap U_j \ne \mt$, then there is a permutation $\pi \in S_{m(i)}$
so that $U_{ik} \cap U_{jl} \ne \mt$ if and only if $l = \pi(k)$.
\item The corresponding open cover of $F^0$ and $F^1$ obtained
by setting $V_i = \tau(U_i)$ and $V_{ik} = \gamma_i(U_{ik})$ also
satisfies (C5).
\end{enumerate}
\end{definition}

Note that $\{U_{ik} : 1 \le i \le n,\, 1 \le k \le m(i) \}$ is an open cover of $E^1$.
The property (C1) is easily achieved in general.
It implies that $m(i)=m(j)$ if $U_i$ and $U_j$ intersect.
Thus this is a locally constant function.

(C2) is just a restatement of local equivalence.
The fact that $\gamma_i$ intertwines
the source maps $s_E$ and $s_F$ and
takes $U_{ik}$ onto $V_{ik}$ means that
\[
 \gamma_i |_{U_{ik}} = (s_F|_{V_{ik}})^{-1} s_E|_{U_{ik}}
 \qfor 1 \le i \le n \AND 1 \le k \le m .
\]

Property (C3) is an expression of $\dim(E^0) \le 1$.
Properties (C4) and (C5) may be achievable in general, but our proof
will use the low dimension in a key way.

Property (C5) means that $U_{ik}$ intersects
exactly one of the $U_{jl}$, not some more complicated mixing.
This seems essential in order to be able to understand what is happening.
If the sets $U_i \cap U_j$ were connected, then this would be automatic.
But in the dimension 0 case, that can't happen.

The open cover of $F^0$ and $F^1$ obtained as defined in (C6)
is easily seen to satisfy (C1,3,4).  However property (C5) is not
automatically transferred.

\begin{theorem}\label{T:topo}
Suppose that $E$ is compact topological graph where the covering
dimension of $E^0$ is at most 1 and that $E$ is locally conjugate to
another topological graph $F$ via a homeomorphism $\tau$.
Then there exists an admissible open cover $(U_i)_{1 \le i \le n}$
of $E^0$.
\end{theorem}

\begin{proof}
The definition of local equivalence shows that for each $u \in E^0$,
there is an open neighbourhood $U$ and a homeomorphism $\gamma$
from  $s_E^{-1}(U)$ onto $s_F^{-1}(\tau U)$ such that
$s_F \gamma = \tau s_E|_{s_E^{-1}(U)}$ and
$r_F \gamma_i = \tau r_E|_{s_E^{-1}(U)}$.
Now $s_E^{-1}(u)$ is finite, say $\{e_j : 1 \le j \le m\}$, and $s_E$ is a local
homeomorphism.
Thus there is an open set $\tilde U \subset U$ containing $u$
and disjoint open sets $V_j$ containing $e_j$ so that
$s_E|_{V_j}$ is a homeomorphism onto $\tilde U$.
Let $U' = \tilde U \bsl s_E(E^1 \bsl \bigcup_{j=1}^m V_j )$.
This is open, contains $u$, and $s_E^{-1}(U') = \bigcup_{j=1}^m V'_j$
where $V'_j = s_E^{-1}(U') \cap V_j$, and $s_E$ maps each $V'_j$
homeomorphically onto $U'$.
The restriction of $\gamma$ to $s_E^{-1}(U')$ is the desired intertwining map.

By the compactness of $E^0$, there is a finite subcover $(U'_i)_{i \leq n}$
of sets constructed in the previous paragraph.
Thus this cover satisfies (C1) and (C2).
Since $E^0$ has covering dimension at most 1, this cover can
be further refined so that (C3) holds as well.

Now we make another modification to obtain (C4).
Let $f_i$ be a partition of unity in $\rC(E^0)$ such that $\supp(f_i) \subset U'_i$.
Observe that at most two of the $f_i(u)$ can be non-zero for any $u \in E^0$.
Thus for each $u\in E^0$, there is some $i$ with $f_i(u) > 1/3$.
Set $U_i = \{u : f_i(u) > 1/3\}$.
This is a finer cover of $E^0$ with $\ol{U_i} \subset U'_i$.

Fix $i \ne j$, and consider
$u \in \big( \ol{U_i} \bsl U_j \big) \cap \big( \ol{U_j} \bsl U_i \big)$.
Then
\[ \tfrac1 3 \le f_i(u) \le \tfrac1 3 \qand \tfrac1 3 \le f_j(u) \le \tfrac1 3 .\]
So $f_i(u)=f_j(u) = 1/3$ and $f_p(u) = 0$ for $i \ne p \ne j$.
This contradicts $\sum f_i = 1$, so the intersection is empty
and (C4) holds and properties (C1-3) are not affected.

Suppose that $U_i \cap U_j \ne \mt$.
By Tietze's extension theorem, there is a continuous
function $f$ in $\rC(E^0,[0,1])$ such that $f|_{\ol{U_i} \bsl U_j}=1$
and $f|_{\ol{U_j} \bsl U_i}=0$.

For each $1 \le k,l \le m(i)$, let $W_{kl} = U_{ik} \cap U_{jl}$.
Observe that for fixed $k$, $U_{ik} \cap s_E^{-1}(U_i\cap U_j)$
is the disjoint union of $W_{kl}$ for $1 \le l \le m(i)$.
So $\{ s_E(W_{kl}) : 1 \le l \le m(i) \}$ partitions $U_i\cap U_j$
into disjoint open sets.
Thus $\{ s_E(W_{kl}) : 1 \le k, l \le m(i) \}$ partitions $U_i\cap U_j$
into disjoint open sets.
Likewise the sets $V_{ik} = \gamma_i(U_{ik})$ and $V_{jl} = \gamma_j(U_{jl})$
determine open sets $X_{kl} = V_{ik} \cap V_{jl}$.
Thus we obtain a partition of $U_i \cap U_j$ into disjoint open
sets $\{ Y_r : 1 \le r \le r(i,j) \}$ determined by
\[ \{ s_E(W_{kl}), \tau^{-1} s_F(X_{kl}) : 1 \le k, l \le m(i) \}  .\]

Using the function $f$ constructed above, define
\[
 U'_i = \{ u \in U_i : f(u) > .6 \} \qand
 U'_j = \{ u \in U_j : f(u) < .4 \},
\]
and set
\[
 Y'_r = \{ u \in Y_r : .3 < f(u) < .7 \} \FOR 1 \le r \le r(i,j).
\]
Replace $U_i$ and $U_j$ in the cover by these new sets.
For each $r \le r(i,j)$ and $k \le m(i)$,
set $Y_{rk} = s_E^{-1}(Y'_r) \cap U_{ik}$.
Note that if $i \ne p \ne j$, then $U_p \cap U'_i = U_p \cap U_i$,
$U_p \cap U'_j = U_p \cap U_j$ and $U_p \cap Y'_r = \mt$.
In addition, $U'_i \cap U'_j = \mt$.
It is clear by construction that $U_{ik} \cap W_{rl}$ is empty except for $l=k$.
Likewise $W_{rk} \cap U_{jl}$ is empty except for a single choice of $l = \pi(k)$.
The same argument shows that the corresponding cover of $F$ satisfies (C5).

Treat all intersecting pairs $U_i$ and $U_j$ from the original cover in this
manner, and the new cover constructed satisfies (C5).  This does not affect
(C1-3).  If necessary, go through the construction to get (C4) again.
Reducing each open set does not affect (C1-3,5).  So we are done.
\end{proof}

It is clear that $m(i)=m(j)$ if $U_i$ and $U_j$ intersect.
So $E^0$ splits into a finite number of clopen subsets on which
$m(i)$ is constant.
Clearly it suffices to deal which each set separately.
So we may suppose that $m(i)=m$ is constant.

Let $\U = (U_i)_{1 \le i \leq n}$ be an admissible open cover of $E^0$.
Note that if $U_i \cap U_j \neq \emptyset$, then there is a unique
permutation $\pi^E_{i,j} \in S_m$ defined by $\pi^E_{i,j}(k)=l$
if and only if $U_{ik} \cap U_{jl} \neq \emptyset$.
We denote $\Pi_{E,\U}$ the function defined
on a subset of $\{1,\dots, n\} \times \{1,\dots, n\}$
by $\Pi_{E,\U}(i,j)=\pi^E_{i,j}$ when $U_i$ and $U_j$ intersect.
Obviously, $\pi^E_{j,i}=(\pi^E_{i,j})^{-1}$.

First we record an easy lemma to set the stage.

\begin{lemma} \label{L:id}
Suppose that $E$ and $F$ are two locally conjugate compact topological graphs,
and that $E^0$ has dimension at most 1.
Let $\U = (U_i)_{1 \le i \leq n}$ denote an admissible open cover
of $E^0$ with $m(i)=m$ for $1 \le i \le n$, and let $\V$ be the corresponding
cover for $F$.
We assume that $\Pi_{E,\U}=\Pi_{F,\V}$.
Then $\X(E)$ and $\X(F)$ are unitarily equivalent as C*-correspondences;
and $\T(E)_+$ and $\T(F)_+$ are completely isometrically isomorphic.
\end{lemma}

\begin{proof}
The map $\gamma = \bigcup_{1 \le i \le n} \gamma_i$ is a
well defined homeomorphism of $E^1$ onto $F^1$ because
the assumption $\Pi_{E,\U}=\Pi_{F,\U}$ means that $\gamma_i$
and $\gamma_j$ agree on $s_E^{-1}(U_i\cap U_j)$.
Moreover $s_F \gamma = \tau s_E$ and
 $r_F \gamma = \tau r_E$.
This induces a unitary isomorphism $\Gamma$ from $\X(E)$
onto $\X(F)$ given by $\Gamma(x) = x \circ \gamma^{-1}$.
To see that $\Gamma$ is a bimodule map, note that
$\tau$ induces a $*$-isomorphism $\rho$ of $\rC(E^0)$
onto $\rC(F^0)$ by $\rho(f) = f\circ \tau^{-1}$.  Moreover,
$\tau^{-1} s_F = s_E \gamma_i^{-1}$ and $\tau^{-1} r_F = r_E \gamma_i^{-1}$.
Thus if $g\in\rC(E^0)$ and  $x \in \X(E)$, then
\begin{align*}
 \Gamma(x\cdot g) &= (x \circ\gamma_i^{-1}) (g\circ s_E \gamma_i^{-1}) \\&=
 \Gamma(x)  (g\circ\tau^{-1}s_F) = \Gamma(x)\cdot \rho(g) .
\end{align*}
Similarly, $\Gamma(g\cdot x) = \rho(g)\cdot \Gamma(x)$.
Finally this map is isometric because if $v\in E^0$ and
$s_E^{-1}(v) = \{ e_i : 1 \le i \le m \}$, then
\[ s_F^{-1}(\tau v) = \{ f_i = \gamma(e_i) : 1 \le i \le m \} .\]
So
\begin{align*}
 \ip{\Gamma(x),\Gamma(x)}(\tau v) &= \sum_{i=1}^m |\Gamma(x)(f_i)|^2
 \\&= \sum_{i=1}^m |x(e_i)|^2 =  \ip{x,x}(v)
\end{align*}
Thus $\ip{\Gamma(x),\Gamma(x)} = \rho( \ip{x,x})$.
It follows that this is a unitary equivalence of C*-correspondences.
Hence $\T(E)_+$ and $\T(F)_+$ are completely isometrically isomorphic.
\end{proof}

Now we modify this proof to get the key step in the proof of the converse.

\begin{lemma}\label{L:flip}
Suppose that $E$ and $F$ are two locally conjugate compact topological graphs,
and that $E^0$ has dimension at most 1.
Let $\U = (U_i)_{1 \le i \leq n}$ denote an admissible open cover
of $E^0$ with $m(i)=m$ for $1 \le i \le n$, and let $\V$ be the corresponding
cover for $F$.
We assume that $\Pi_{E,\U}=\Pi_{F,\V}$ where they're defined,
except for a single pair $i_0,j_0$ at which the permutations differ
by a single transposition $(k_0l_0)$:
\[
 \pi^E_{i_0,j_0} = \pi^F_{i_0,j_0} \circ (k_0l_0) .
\]
Then $\X(E)$ and $\X(F)$ are unitarily equivalent as C*-correspondences.
\end{lemma}

\begin{proof}
Let $\V = \{V_i = \tau U_i : 1 \le i \le n\}$ be the corresponding cover of $F^0$,
and let $V_{ik} = \gamma_i(U_{ik})$ for $1 \le i \le n$ and $1 \le k \le m$.
Since the labelling is arbitrary, we can suppose without
loss of generality that $i_0=k_0=1$ and
$j_0=l_0=2$ and that $\pi^E_{1,2}=id$  and so $\pi^F_{1,2} =(12)$;
while $\pi^F_{i,j} = \pi^E_{i,j}$ when $\{i,j\} \ne \{1,2\}$.
Thus the sets
$U_{11}\cap U_{21}$, $U_{12}\cap U_{22}$,
$V_{11} \cap V_{22}$ and $V_{12} \cap V_{21}$ are all non-empty.

Except for this one flip, we have that $V_{ik} \cap V_{jl} \ne \mt$
if and only if $U_{ik} \cap U_{jl} \ne \mt$ if and only if $\pi^E_{i,j}(k)=l$.
Thus
\[
 \gamma_i|_{U_{ik} \cap U_{jl}} =
 (s_F|_{V_{ik} \cap V_{jl}})^{-1} s_E|_{U_{ik} \cap U_{jl}} =
 \gamma_j|_{U_{ik} \cap U_{jl}} .
\]
To describe what happens on the remaining four intersections, let
\[
 \sigma = (s_F|_{V_{12} \cap V_{21}})^{-1} s_F|_{V_{11} \cap V_{22}}
\]
be the canonical homeomorphism of $V_{11} \cap V_{22}$
onto $V_{12} \cap V_{21}$.  Then
\[
 \gamma_1|_{U_{11}\cap U_{21}} =
 (s_F|_{V_{11} \cap V_{22}} )^{-1} s_E|_{U_{11}\cap U_{21}}
\]
whereas
\[
 \gamma_2|_{U_{11}\cap U_{21}} =
 (s_F|_{V_{12} \cap V_{21}} )^{-1} s_E|_{U_{11}\cap U_{21}} =
 \sigma \gamma_1|_{U_{11}\cap U_{21}} .
\]
Similarly, $\gamma_1|_{U_{12}\cap U_{22}}$ is a homeomorphism onto
$V_{12} \cap V_{21}$, while
\[
 \gamma_2|_{U_{12}\cap U_{22}} =
 \sigma^{-1} \gamma_1|_{U_{12}\cap U_{22}}
\]
is a homeomorphism onto $V_{11} \cap V_{22}$.

Using property (C4) and Tietze's extension theorem, there exists
a continuous function
$g \in \rC(F^0, [0,\pi/2])$ such that
$g|_{\ol{V}_1 \bsl V_2}=0$ and
$g|_{\ol{V}_2 \bsl V_1}=\pi/2$.
Let $h=g\circ s_F$.
For  $x \in \X(F)$ and $a \in F^1$, we define $\Gamma(x) \in \X(F)$
by setting $\Gamma(x)(a)$ to be
\begin{alignat*}{2}
 &x \gamma_i^{-1}(a)
    \hspace{4cm} a \in V_{ik} ,\ a \not\in V_{11} \cap V_{22} &\OR V_{12} \cap V_{21}
    \\[.5ex]
 &\cos(h(a)) x \gamma_1^{-1}(a)  + \sin(h(a)) x \gamma_2^{-1}(a)
    &a \in V_{11} \cap V_{22}\\[.5ex]
 &e^{2ih(a)} \big(\cos(h(a)) x \gamma_1^{-1}(a)  -
  \sin(h(a)) x \gamma_2^{-1}(a) \big)
    &a\in V_{12} \cap V_{21}
\end{alignat*}
By the calculations in the previous paragraph,
$\Gamma(x)$ is well defined on $F^1$.

It is easy to verify that $\Gamma(x)$ is a continuous function on $F^1$.
As before, there is an induced $*$-isomorphism $\rho$ between
$\rC(E^0)$ and $\rC(F^0)$ given by sending $f\in \rC(E^0)$ to
$\rho(f) = f \circ \tau^{-1}$.  The relations
$\tau^{-1} s_F = s_E \gamma_i^{-1}$ and $\tau^{-1} r_F = r_E \gamma_i^{-1}$
allow us to show that $\Gamma$ is a bimodule map:
if $f\in\rC(E^0)$, $x \in \X(E)$ and $a\in F^1\bsl
\big((V_{11} \cap V_{22}) \cap (V_{12} \cap V_{21}) \big) $, then
the argument is as in Lemma~\ref{L:id}.
While if $a \in V_{11} \cap V_{22}$, then
\begin{align*}
 \Gamma(x\cdot f)(a) &=
 \big(\cos(h(a)) x \gamma_1^{-1}(a)  + \sin(h(a)) x \gamma_2^{-1}(a) \big)
 f(s_E \gamma_i^{-1}(a)) \\&=
 \Gamma(x)(a) f(\tau^{-1}s_F(a)) = \Gamma(x)\cdot \rho(f) (a) .
\end{align*}
A similar calculation works on $V_{12} \cap V_{21}$, so $\Gamma$
is a right module map.  Similarly, it is a left module map.

Finally we show that $\Gamma$ is isometric.
We will prove that for every $v \in E^0$,
$\ip{\Gamma(x),\Gamma(x)}(\tau v) = \ip{x,x}(v)$.
It is clear from the definition of $\Gamma$ and the
argument in Lemma~\ref{L:id} that the only place a problem
might arise is when $v \in U_1 \cap U_2$.
If $s_E^{-1}(v) = \{ e_i : 1\le i \le m\}$, then
$s_F^{-1}(\tau v) = \{ f_i := \gamma_1(e_i): 1\le i \le m\}$.
So
\[
  \ip{x,x}(v) = \sum_{i=1}^m |x(e_i)|^2
\]
and
\begin{align*}
  \ip{\Gamma(x),\Gamma(x)}(\tau v) &=
  \sum_{i=1}^m |\Gamma(x)(\gamma_1(e_i))|^2  \\&=
  \sum_{i=1}^2 |\Gamma(x)(f_i)|^2 + \sum_{i=3}^m |x(e_i)|^2 .
\end{align*}
Thus we can ignore the terms $3 \le i \le m$.
We compute\\[1ex]
$|\Gamma(x)(f_1)|^2 + |\Gamma(x)(f_2)|^2 = $
\begin{align*}
   &=
\big| \cos(h(f_1)) x \gamma_1^{-1}(f_1)  + \sin(h(f_1)) x \gamma_2^{-1}(f_1) \big|^2
 \\& \quad +
 \big|e^{2ih(f_2)} \big(\cos(h(f_2)) x \gamma_1^{-1}(f_2)  -
  \sin(h(f_2)) x \gamma_2^{-1}(f_2) \big) \big|^2 \\ &=
 \cos^2(g(\tau v)) |x(e_1)|^2  + \sin(2g(\tau v)) \re \ol{x(e_1)}x(e_2)
  + \sin^2(g(\tau x)) |x(e_2)|^2
 \\& \quad +
 \cos^2(g(\tau v)) |x(e_2)|^2  - \sin(2g(\tau v)) \re \ol{x(e_1)}x(e_2)
 + \sin^2(g(\tau x)) |x(e_1)|^2
 \\&= |x(e_1)|^2 + |x(e_2)|^2
\end{align*}
Hence $\Gamma$ is a unitary.

It follows that $\X(E)$ and $\X(F)$ are unitarily equivalent as C*-corr\-es\-pondences.
\end{proof}

Now we make the routine extension to the general result.

\begin{theorem}\label{T:converse}
Let $E=(E^0,E^1,s_E,r_E)$ and $F=(F^0,F^1,s_F,r_F)$ be two compact
topological graphs with $\dim E^0 \le 1$.
Then the following are equivalent:
\begin{enumerate}[\em(i)]
\item $\T (E)_+$ and $\T (F)_+$ are algebraically isomorphic,
\item $E$ and $F$ are locally conjugate,
\item $\T (E)_+$ and $\T (F)_+$ are completely isometrically isomorphic.
\end{enumerate}
\end{theorem}

\begin{proof}
We established (i) implies (ii) in Theorem~\ref{T:main}.
As (iii) clearly implies (i), it remains to establish (ii) implies (iii).

Let $\U$ be an admissible cover. Let us write
$\sigma_{i,j}= (\pi_{i,j}^F)^{-1} \pi_{i,j}^E$ for $1 \le i,j \leq n$.
Fix $k$ and set $\pi_{i,j}^E(k)=l$, $k'= \pi_{i,j}^{F\,-1}(l)$
and $l'=\pi_{i,j}^E(k')$; so that  $\sigma_{ij}(k)=k'$.
Then $\gamma_j$ carries $U_{ik}\cap U_{il}$ onto $V_{ik'}\cap V_{jl}$,
which is carried by $\gamma_i^{-1}$ onto $U_{ik'} \cap U_{jl'}$.
By the local equivalence (C2), it follows that
\[
 s_E \gamma_i^{-1}\gamma_j |_{U_{ik}\cap U_{il}} = s_E|_{U_{ik}\cap U_{il}}
 \AND
 r_E \gamma_i^{-1}\gamma_j |_{U_{ik}\cap U_{il}} = r_E|_{U_{ik}\cap U_{il}} .
\]
It follows that we can build a new topological graph by changing $\pi^E_{ij}$
by the transposition $(l,l')$ and leaving the rest the same.  That is, we
declare that $U_{ik}$ intersects $U_{jl'}$ and $U_{ik'}$ intersects $U_{jl}$,
leaving all other intersections unchanged.  This defines a new topological
graph, which by Lemma~\ref{L:flip} determines a C*-correspondence
which is unitarily equivalent to $\X(E)$.

Repeated use of this procedure a finite number of times yields
an equivalent topological graph $\tilde E$ for which $\Pi_{\tilde E,\U} = \Pi_{F,\V}$.
By Lemma~\ref{L:id}, the two C*-correspondences $\X(\tilde E)$ and $\X(F)$ are
unitarily equivalent.  It follows that $\X(E)$ and $\X(F)$ are unitarily equivalent,
whence $\T (E)_+$ and $\T (F)_+$ are completely isometrically isomorphic.
\end{proof}



\begin{thebibliography}{99}

\bibitem{Ar} W. Arveson,
\textit{Operator algebras and measure preserving automorphisms},
Acta Math.\ \textbf{118} (1967), 95--109.

\bibitem{AJ} W. B. Arveson and K. B. Josephson,
\textit{Operator algebras and measure preserving automorphisms. II.}
J. Funct.\ Anal.\ \textbf{4} (1969) 100--134.

\bibitem{DK2} K. Davidson and E. Katsoulis,
\textit{Isomorphisms between topological conjugacy algebras},
J. Reine Angew.\ Math. \textbf{621} (2008), 29--51.

\bibitem{DK} K.R. Davidson and E.G. Katsoulis,
\textit{Operator algebras for multivariable dynamics},
Memoirs Amer.\ Math.\ Soc., to appear.

\bibitem{HH}  D. Hadwin and  T. Hoover,
\textit{Operator algebras and the conjugacy of transformations.},
J. Funct.\ Anal.\ \textbf{77} (1988), 112--122.

\bibitem{KK} E. Katsoulis and D. W. Kribs,
\textit{Isomorphisms of algebras associated with directed graphs},
Math.\ Ann.\ \textbf{330} (2004), 709--728.

\bibitem{K} T. Katsura,
\textit{A class of C*-algebras generalizing both graph algebras and
homeomorphism C*-algebras, I},
Trans.\ Amer.\ Math.\ Soc. \textbf{356}, 4287--4322.

\bibitem{K2} T. Katsura,
\textit{A class of C*-algebras generalizing both graph algebras
and homeomorphism C*-algebras. II. Examples},
Internat.\ J. Math.\ \textbf{17} (2006), 791--833.

\bibitem{KP} D. Kribs and S. Power,
\textit{Free semigroupoid algebras},
J. Ramanujan Math.\ Soc.\ \textbf{19} (2004), 117--159.

\bibitem{MS1} P.S. Muhly and B. Solel,
\textit{Tensor algebras over C*-correspondences: Representations,
Dilations, and C*-envelopes}.
J. Funct.\ Anal. \textbf{158} (1998), 389--457.

\bibitem{MS2} P. Muhly and  B. Solel,
\textit{Tensor algebras, induced representations, and the Wold
decomposition},
Can.\ J. Math.\ {\bf 51}  (1999), 850--880.

\bibitem{Pet} J. Peters,
\textit{Semicrossed products of $C\sp *$-algebras},
J. Funct.\ Anal.\ \textbf{59} (1984), 498--534.

\bibitem{Pim}M. Pimsner,
\textit{A class of $C\sp *$-algebras generalizing both
Cuntz-Krieger algebras and crossed products by $\bZ$},
\textsf{Free probability theory} (D. Voiculescu, ed.), pp. 189--212,
Fields Inst.\ Commun., \textbf{12},
Amer.\ Math.\ Soc., Providence, RI, 1997.

\bibitem{Pow} S. Power,
\textit{Classification of analytic crossed product algebras},
Bull.\ London Math.\ Soc.\ \textbf{24} (1992), 368--372.

\bibitem{R} I. Raeburn,
\textit{Graph Algebras},
CBMS Reg.\ Conf.\ Ser.\ Math.\ \textbf{103},
Amer.\ Math.\ Soc., 2005.

\bibitem{Sol} B. Solel,
\textit{You can see the arrows in a quiver operator algebra}
J. Aust.\ Math.\ Soc.\ \textbf{77} (2004), 111--122.

\end{thebibliography}
\end{document}